\documentclass[12pt,letterpaper,american,english]{amsart}
\usepackage{charter}
\usepackage{helvet}
\usepackage{courier}
\usepackage[T1]{fontenc}
\usepackage[utf8x]{inputenc}
\usepackage{mathrsfs}
\usepackage{amsthm}
\usepackage{amssymb}
\usepackage[all]{xy}

\makeatletter

\pdfpageheight\paperheight
\pdfpagewidth\paperwidth

\numberwithin{equation}{section}
\numberwithin{figure}{section}
  \theoremstyle{plain}
  \newtheorem*{cor*}{\protect\corollaryname}
\theoremstyle{plain}
\newtheorem{thm}{\protect\theoremname}[section]
  \theoremstyle{definition}
  \newtheorem{defn}[thm]{\protect\definitionname}
  \theoremstyle{remark}
  \newtheorem*{rem*}{\protect\remarkname}
  \theoremstyle{plain}
  \newtheorem{prop}[thm]{\protect\propositionname}
  \theoremstyle{plain}
  \newtheorem*{fact*}{\protect\factname}
  \theoremstyle{plain}
  \newtheorem{cor}[thm]{\protect\corollaryname}

\usepackage{verbatim} 
\usepackage{eucal}
\usepackage{amsthm}
\usepackage[all]{xy}
\usepackage{manfnt}
\usepackage{float}
\usepackage{xr}
\SelectTips{cm}{}
\xyoption{line}
\usepackage{hyperref}
\makeatletter \newcommand{\xyR}[1]{%
\makeatletter \xydef@\xymatrixrowsep@{#1} \makeatother }
\makeatletter \newcommand{\xyC}[1]{%
\makeatletter \xydef@\xymatrixcolsep@{#1} \makeatother }
\mathcode`\:="603A 
\DeclareSymbolFont{rsfs}{U}{rsfs}{m}{n}
\DeclareSymbolFontAlphabet{\mathrf}{rsfs}

\makeatother

\usepackage{babel}
  \addto\captionsamerican{\renewcommand{\corollaryname}{Corollary}}
  \addto\captionsamerican{\renewcommand{\definitionname}{Definition}}
  \addto\captionsamerican{\renewcommand{\factname}{Fact}}
  \addto\captionsamerican{\renewcommand{\propositionname}{Proposition}}
  \addto\captionsamerican{\renewcommand{\remarkname}{Remark}}
  \addto\captionsamerican{\renewcommand{\theoremname}{Theorem}}
  \addto\captionsenglish{\renewcommand{\corollaryname}{Corollary}}
  \addto\captionsenglish{\renewcommand{\definitionname}{Definition}}
  \addto\captionsenglish{\renewcommand{\factname}{Fact}}
  \addto\captionsenglish{\renewcommand{\propositionname}{Proposition}}
  \addto\captionsenglish{\renewcommand{\remarkname}{Remark}}
  \addto\captionsenglish{\renewcommand{\theoremname}{Theorem}}
  \providecommand{\corollaryname}{Corollary}
  \providecommand{\definitionname}{Definition}
  \providecommand{\factname}{Fact}
  \providecommand{\propositionname}{Proposition}
  \providecommand{\remarkname}{Remark}
\providecommand{\theoremname}{Theorem}

\begin{document}

\title{Steenrod coalgebras III. The fundamental group}

\author{Justin R. Smith}



\curraddr{Department of Mathematics\\
Drexel University\\
Philadelphia,~PA 19104}

\email{jsmith@drexel.edu}

\urladdr{http://vorpal.math.drexel.edu}

\maketitle
\global\long\def\ring{\mathbb{Z}}
\global\long\def\integers{\mathbb{Z}}
\global\long\def\betabar{\bar{\beta}}
 \global\long\def\desusp{\downarrow}
\global\long\def\susp{\uparrow}
\global\long\def\cobar{\mathcal{F}}
\global\long\def\coend{\mathrm{CoEnd}}
\global\long\def\ainfty{A_{\infty}}
\global\long\def\coassoc{\mathrm{Coassoc}}
\global\long\def\trm{\mathrm{T}}
\global\long\def\tfr{\mathfrak{T}}
\global\long\def\tabbr{\hat{\trm}}
\global\long\def\Tabbr{\hat{\tfr}}
\global\long\def\afr{\mathfrak{A}}
\global\long\def\homz{\mathrm{Hom}_{\ring}}
\global\long\def\zend{\mathrm{End}}
\global\long\def\rs#1{\mathrm{R}S_{#1 }}
\global\long\def\forgetful#1{\lceil#1\rceil}
\global\long\def\highprod#1{\bar{\mu}_{#1 }}
\global\long\def\slength#1{|#1 |}
\global\long\def\barcs{\bar{\mathcal{B}}}
\global\long\def\ubarcs{\mathcal{B}}
\global\long\def\zs#1{\ring S_{#1 }}
\global\long\def\homzs#1{\mathrm{Hom}_{\ring S_{#1 }}}
\global\long\def\zpi{\mathbb{Z}\pi}
\global\long\def\D{\mathfrak{D}}
\global\long\def\ahat{\hat{\mathfrak{A}}}
\global\long\def\cbar{{\bar{C}}}
\global\long\def\cf#1{C(#1 )}
\global\long\def\ddelta{\dot{\Delta}}
\global\long\def\dimlimiter{\triangleright}
\global\long\def\coalgcat{\mathrf S_{0}}
\global\long\def\hcoalgcat{\mathrf{S}}
\global\long\def\ircoalgcat{\mathrf I_{0}}
\global\long\def\bircoalgcat{\mathrf{I}_{0}^{+}}
\global\long\def\hircoalgcat{\mathrf I}
\global\long\def\dcoalgcat{\mathrm{ind}-\coalgcat}
\global\long\def\chaincat{\mathbf{Ch}}
\global\long\def\coll{\mathrm{Coll}}
\global\long\def\bchaincat{\mathbf{Ch}_{0}}
\global\long\def\ilimit{\varprojlim\,}
\global\long\def\bigboxtimes{\mathop{\boxtimes}}
\global\long\def\dlimit{\varinjlim\,}
\global\long\def\coker{\mathrm{{coker}}}
\global\long\def\icoalgcat{\mathrm{pro}-\mathrf{S}_{0}}
\global\long\def\iircoalgcat{\mathrm{pro-}\ircoalgcat}
\global\long\def\dircoalgcat{\mathrm{ind-}\ircoalgcat}
\global\long\def\core#1{\left\langle #1\right\rangle }
\global\long\def\ilimitder{\varprojlim^{1}\,}
\global\long\def\pcoalg#1#2{P_{\mathcal{#1}}(#2) }
\global\long\def\pcoalgf#1#2{P_{\mathcal{#1}}(\forgetful{#2}) }
\global\long\def\coequalizer{\mathop{\mathrm{coequalizer}}}

\global\long\def\mainoperad{\mathcal{H}}
\global\long\def\cone#1{\mathrm{Cone}(#1)}

\global\long\def\im{\operatorname{im}}

\global\long\def\lcell{L_{\mathrm{cell}}}
\global\long\def\ccoalgcat{\mathrf S_{\mathrm{cell}}}

\global\long\def\fc#1{\mathrm{hom}(\bigstar,#1)}
\global\long\def\coS{\mathbf{coS}}
\global\long\def\cocell{\mathbf{co}\ccoalgcat}

\global\long\def\ccoalgcat{\mathrf S_{\mathrm{cell}}}

\global\long\def\spaces{\mathbf{SS}}

\global\long\def\pgam{\tilde{\Gamma}}
\global\long\def\pz{\tilde{\integers}}

\global\long\def\moore#1{\{#1\}}

\global\long\def\ints{\mathbb{Z}}

\global\long\def\finite{\mathcal{F}}

\global\long\def\finiteop{\finite^{\mathrm{op}}}

\global\long\def\syms{\mathbf{SS}}

\global\long\def\ordered{\mathbf{\Delta}}

\global\long\def\sets{\mathbf{Set}}

\global\long\def\colim{\operatorname{colim}}

\newdir{ >}{{}*!/-5pt/@{>}}

\global\long\def\treal#1{\mathcal{T}(\bigstar,#1)}

\global\long\def\rats{\mathbb{Q}}

\global\long\def\img{\operatorname{im}}

\global\long\def\tmap#1{\mathrm{T}_{#1}}

\global\long\def\Tmap#1{\mathfrak{T}_{#1}}

\global\long\def\glist#1#2#3{#1_{#2},\dots,#1_{#3}}

\global\long\def\blist#1#2{\glist{#1}1{#2}}

\global\long\def\enlist#1#2{\{\blist{#1}{#2}\}}

\global\long\def\tlist#1#2{\tmap{\blist{#1}{#2}}}

\global\long\def\Tlist#1#2{\Tmap{\blist{#1}{#2}}}

\global\long\def\nth#1{\mbox{#1}^{\mathrm{th}}}

\global\long\def\tunder#1#2{\tmap{\underbrace{{\scriptstyle #1}}_{#2}}}

\global\long\def\Tunder#1#2{\Tmap{\underbrace{{\scriptstyle #1}}_{#2}}}

\global\long\def\tunderi#1#2{\tunder{1,\dots,#1,\dots,1}{#2^{\mathrm{th}}\ \mathrm{position}}}

\global\long\def\Tunderi#1#2{\Tunder{1,\dots,#1,\dots,1}{#2^{\mathrm{th}}\,\mathrm{position}}}

\global\long\def\chaincat{\mathbf{Ch}}

\global\long\def\chaincatp{\chaincat_{0}}

\global\long\def\simpc{\mathbf{SC}}

\global\long\def\s{\mathfrak{S}}

\global\long\def\pco{P_{\s}}

\global\long\def\lco{L_{\s}}

\global\long\def\ns#1{\mathcal{N}^{#1}}

\global\long\def\cfn#1{N(#1)}

\global\long\def\freeop{\mathcal{F}}

\global\long\def\kerftos{\mathscr{R}}

\global\long\def\comm{\mathbf{Commute}}

\global\long\def\steen{\mathscr{S}}

\global\long\def\arity{\operatorname{arity}\,}

\global\long\def\nfc#1{\mathrm{hom}_{\steen}(\bigstar,#1)}

\global\long\def\ss{\mathbf{S}}

\global\long\def\ssz{\ss_{0}}

\global\long\def\dcat{\mathbf{D}}

\global\long\def\ords{\mathbf{\Delta}_{+}}

\global\long\def\sd{\mathfrak{f}}

\global\long\def\ds{\mathfrak{d}}

\date{\today}
\begin{abstract}
In this note, we extend earlier work by showing that if $X$ and $Y$
are delta-complexes (i.e. simplicial sets without degeneracy operators),
a\emph{ }morphism $g:\cfn X\to\cfn Y$ of Steenrod coalgebras (normalized
chain-complexes equipped with extra structure) induces one of 2-skeleta
$\hat{g}:X_{2}\to Y_{2}$, inducing a homomorphism $\pi_{1}(\hat{g}):\pi_{1}(X)\to\pi_{1}(Y)$
that is an isomorphism if $g$ is an isomorphism. This implies a corresponding
conclusion for a morphism $g:C(X)\to C(Y)$ of Steenrod coalgebras
on \emph{unnormalized} chain-complexes of \emph{simplicial sets}. 
\end{abstract}

\section{Introduction}

It is well-known that the Alexander-Whitney coproduct is functorial
with respect to simplicial maps. If $X$ is a simplicial set, $C(X)$
is the unnormalized chain-complex and $\rs 2$ is the \emph{bar-resolution}
of $\ints_{2}$ (see \cite{maclane:1975}), it is also well-known
that there is a unique homotopy class of $\ints_{2}$-equivariant
maps (where $\ints_{2}$ transposes the factors of the target) 
\[
\xi_{X}:\rs 2\otimes C(X)\to C(X)\otimes C(X)
\]
cohomology, and that this extends the Alexander-Whitney diagonal.
We will call such structures, Steenrod coalgebras and the map $\xi_{X}$
the Steenrod diagonal.

With some care (see appendix~A of \cite{smith-steenrod2}), one can
construct $\xi_{X}$ in a manner that makes it \emph{functorial} with
respect to simplicial maps although this is seldom done since the
\emph{homotopy class} of this map is what is generally studied. \foreignlanguage{english}{The
paper \cite{smith-steenrod2} showed that:}
\begin{cor*}
\ref{cor:cellular-determines-pi1}. If $X$ and $Y$ are simplicial
complexes (simplicial sets without degeneracies whose simplices are
uniquely determined by their vertices), any purely algebraic chain
map of normalized chain complexes\emph{ }
\[
f:\cfn X\to\cfn Y
\]
 that makes the diagram 
\begin{equation}
\xymatrix{{\rs 2\otimes\cfn X}\ar[r]^{1\otimes f}\ar[d]_{\xi_{X}} & {\rs 2\otimes\cfn Y}\ar[d]^{\xi_{Y}}\\
{\cfn X\otimes\cfn X}\ar[r]_{f\otimes f} & {\cfn Y\otimes\cfn Y}
}
\label{eq:coproduct-diagram-1}
\end{equation}
commute induces a map of simplicial complexes\emph{
\[
\hat{f}:X\to Y
\]
}If $f$ is an isomorphism then $\hat{f}$ is an isomorphism of simplicial
complexes --- and X and $Y$ are homeomorphic.
\end{cor*}
The note extends that result, slightly, to
\begin{cor*}
\ref{cor:cellular-determines-pi1} If $X$ and $Y$ are delta-complexes,
any morphism of their canonical Steenrod coalgebras (see~proposition~\ref{prop:canonical-steenrod-coalgebra})
\[
g:\cfn X\to\cfn Y
\]
induces a map
\[
\hat{g}:X_{2}\to Y_{2}
\]
of 2-skeleta. If $g$ is an isomorphism then $X_{2}$ and $Y_{2}$
are isomorphic as delta-complexes.
\end{cor*}
and
\begin{cor*}
\ref{cor:pi1-simplicial-set} If $X$ and $Y$ are simplicial sets
and $f:C(X)\to C(Y)$ is a morphism of their canonical Steenrod coalgebras
(see proposition~\ref{prop:canonical-steenrod-coalgebra}) over their
unnormalized chain-complexes, then $f$ induces a map 
\[
\hat{f}:X_{2}\to Y_{2}
\]
of 2-skeleta. If $f$ is an isomorphism, then $\hat{f}$ is a homotopy
equivalence.
\end{cor*}
The author conjectures that the last statement can be improved to
``if $f$ is an isomorphism, then \foreignlanguage{english}{$\hat{f}$
is a homotopy equivalence.''}

The author is indebted to Dennis Sullivan for several interesting
discussions.

\section{Definitions and assumptions\label{sec:Definitions-and-assumptions}}

Given a simplicial set, $X$, $C(X)$ will always denote its \emph{unnormalized}
chain-complex and $N(X)$ its \emph{normalized} one (with degeneracies
divided out).

We consider variations on the concept of simplicial set.
\begin{defn}
\label{def:delta-complexes}Let $\ords$ be the ordinal number category
whose morphisms are order-preserving monomorphisms between them. The
objects of $\ords$ are elements $\mathbf{n}=\{0\to1\to\cdots\to n\}$
and a morphism 
\[
\theta:\mathbf{m}\to\mathbf{n}
\]
 is a strict order-preserving map ($i<k\implies\theta(i)<\theta(j)$).
Then the category of \emph{delta-complexes,} $\dcat$, has objects
that are contravariant functors
\[
\ords\to\mathbf{Set}
\]
to the category of sets. The chain complex of a delta-complex, $X$,
will be denoted $N(X)$.\end{defn}
\begin{rem*}
In other words, delta-complexes are just simplicial sets \emph{without
degeneracies.}

A simplicial set gives rise to a delta-complex by ``forgetting''
its degeneracies --- ``promoting'' its degenerate simplices to nondegenerate
status. Conversely, a delta-complex can be converted into a simplicial
set by equipping it with degenerate simplices in a mechanical fashion.
These operations define functors:\end{rem*}
\begin{defn}
\label{def:sd-ds-functors}The functor
\[
\sd:\ss\to\dcat
\]
 is defined to simply drop degeneracy operators (degenerate simplices
become nondegenerate). The functor
\[
\ds:\dcat\to\ss
\]
equips a delta complex, $X$, with degenerate simplicies and operators
via
\begin{equation}
\ds(X)_{m}=\bigsqcup_{\mathbf{m}\twoheadrightarrow\mathbf{n}}X_{n}\label{eq:ds-functor}
\end{equation}
for all $m>n\ge0$.\end{defn}
\begin{rem*}
The functors $\sd$ and $\ds$ were denoted $F$ and $G$, respectively,
in \cite{rourke-sanderson-delta-complex}. Equation~\ref{eq:ds-functor}
simply states that we add all possible degeneracies of simplices in
$X$ subject \emph{only} to the basic identities that face- and degeneracy-operators
must satisfy. 

Although $\sd$ promotes degenerate simplicies to nondegenerate ones,
these new nondegenerate simplices can be collapsed without changing
the homotopy type of the complex: although the degeneracy operators
are no longer built in to the delta-complex, they still define contracting
homotopies.
\end{rem*}
The definition immediately implies that
\begin{prop}
\label{prop:cx-is-nfx}If $X$ is a simplicial set and $Y$ is a delta-complex,
$C(X)=N(\sd(X))$, $N(\ds(Y))=N(Y)$, and $C(X)=N(\ds\circ\sd(X))$.
\end{prop}
Theorem~1.7 of \cite{rourke-sanderson-delta-complex} shows that
there exists an adjunction:

\begin{equation}
\ds:\dcat\leftrightarrow\ss:\sd\label{eq:ds-sd-adjunction}
\end{equation}
The composite (the \emph{counit} of the adjunction)
\[
\sd\circ\ds:\dcat\to\dcat
\]
maps a delta complex into a much larger one --- that has an infinite
number of (degenerate) simplices added to it. There is a natural inclusion
\[
\iota:X\to\sd\circ\ds(X)
\]
 and a natural map (the \emph{unit} of the adjunction)
\begin{equation}
g:\ds\circ\sd(X)\to X\label{eq:ds-sd-unit}
\end{equation}
The functor $g$ sends degenerate simplices of $X$ that had been
``promoted to nondegenerate status'' by $\sd$ to their degenerate
originals --- and the extra degenerates added by $\ds$ to suitable
degeneracies of the simplices of $X$. 

In \cite{rourke-sanderson-delta-complex}, Rourke and Sanderson also
prove: 
\begin{prop}
\label{prop:homotopy-equiv-sd-ds}If $X$ is a simplicial set and
$Y$ is a delta-complex then
\begin{enumerate}
\item $|Y|$ and $|\ds Y|$ are homeomorphic
\item the map $|g|:|\ds\circ\sd(X)|\to|X|$ is a homotopy equivalence.
\item $\sd:H\ss\to H\dcat$ defines an equivalence of categories, where
$H\ss$ and $H\dcat$ are the homotopy categories, respectively, of
$\ss$ and $\dcat$. The inverse is $\ds:H\dcat\to H\ss$. In particular,
if $X$ is a simplicial set, the natural map
\[
g:\ds\circ\sd(X)\to X
\]
is a homotopy equivalence.
\end{enumerate}
\end{prop}
\begin{rem*}
Here, $|*|$ denotes the topological realization functors for $\ss$
and $\dcat$.\end{rem*}
\begin{proof}
The first two statements are proposition~2.1 of \cite{rourke-sanderson-delta-complex}
and statement~3 is theorem~6.9 of the same paper. The final statement
follows from Whitehead's theorem.
\end{proof}

\section{Steenrod coalgebras\label{sec:Steenrod-coalgebras}}

We begin with:
\begin{defn}
\label{def:Steenrod-coalgebra}A \emph{Steenrod coalgebra,} $(C,\delta)$
is a chain-complex $C\in\chaincat$ equipped with a $\ints_{2}$-equivariant
chain-map
\[
\delta:\rs 2\otimes C\to C\otimes C
\]
 where $\ints_{2}$ acts on $C\otimes C$ by swapping factors and
$\rs 2$ is the bar-resolution of $\ints$ over $\zs 2$. A morphism
$f:(C,\delta_{C})\to(D,\delta_{D})$ is a chain-map $f:C\to D$ that
makes the diagram
\[
\xyR{30pt}\xymatrix{{C}\ar[r]^{f}\ar[d]_{\delta_{C}} & {D}\ar[d]^{\delta_{D}}\\
{C\otimes C}\ar[r]_{f\otimes f} & {D\otimes D}
}
\]
commute.
\end{defn}
Steenrod coalgebras are very general --- the underlying coalgebra
need not even be coassociative. The category of Steenrod coalgebras
is denoted $\steen$.

Appendix~A of \cite{smith-steenrod2} shows that:
\begin{prop}
\label{prop:canonical-steenrod-coalgebra}If $X$ is a simplicial
set or delta-complex, then the unnormalized and normalized chain-complexes
of $X$ have a natural Steenrod coalgebra structure, i.e. natural
maps
\begin{align*}
\xi:\rs 2\otimes N(X) & \to N(X)\otimes N(X)\\
\xi:\rs 2\otimes C(X) & \to C(X)\otimes C(X)
\end{align*}
\end{prop}
\begin{rem*}
If $[\,]$ is the 0-dimensional generator of $\rs 2$, the map $\xi([\,]\otimes*):N(X)\to N(X)\otimes N(X)$
is nothing but the Alexander-Whitney coproduct.

The Steenrod coalgebra structure for $N(X)$ is a \emph{natural quotient}
of that for $C(X)$.
\end{rem*}
Here are some computations of this Steenrod coalgebra structure from
\foreignlanguage{american}{appendix~A of \cite{smith-steenrod2}:}
\begin{fact*}
\label{example:e1timesdelta2}If $\Delta^{2}$ is a $2$-simplex,
then

\begin{equation}
\xi([\,]\otimes\Delta^{2})=\Delta^{2}\otimes F_{0}F_{1}\Delta^{2}+F_{2}\Delta^{2}\otimes F_{0}\Delta^{2}+F_{1}F_{2}\Delta^{2}\otimes\Delta^{2}\label{eq:delta-2-coproduct-1}
\end{equation}
--- the standard (Alexander-Whitney) coproduct --- and

\begin{align}
\xi([(1,2)]\otimes\Delta^{2})= & \Delta^{2}\otimes F_{0}\Delta^{2}-F_{1}\Delta^{2}\otimes\Delta^{2}\label{eq:e1timesdelta2-1}\\
 & -\Delta^{2}\otimes F_{2}\Delta^{2}\nonumber 
\end{align}

\end{fact*}
Corollary~4.3 of \cite{smith-steenrod2} proves that:
\begin{cor}
\label{cor:n-simplices-map-to-simplices}Let $X$ be a simplicial
set and suppose 
\[
f:\ns n=\cfn{\Delta^{n}}\to\cfn X
\]
 is a Steenrod coalgebra morphism. Then the image of the generator
$\Delta^{n}\in\cfn{\Delta^{n}}_{n}$ is a generator of $\cfn X_{n}$
defined by an $n$-simplex of $X$.
\end{cor}
We can prove a delta-complex (partial) analogue of corollary~4.5
in \foreignlanguage{american}{\cite{smith-steenrod2}:}
\begin{cor}
\label{cor:cf-gives-simplices}Let $X$ be a delta-complex, let $n\le2$,
and let 
\[
f:\cfn{\Delta^{n}}\to\cfn X
\]
map $\Delta^{n}$ to a simplex $\sigma\in N(X)$ defined by the simplicial-map
$\iota:\Delta^{n}\to X$. Then $f=\cfn{\iota}$.\end{cor}
\begin{proof}
Let
\[
\xi_{i}=\xi(e_{i}\otimes*):N(\Delta^{n})\to N(\Delta^{n})\otimes N(\Delta^{n})
\]
 denote the Steenrod coalgebra structure, where $e_{i}$ is the generator
of $(\rs 2)_{i}$. By hypothesis, the diagram\foreignlanguage{american}{
\[
\xymatrix{{\cfn{\Delta^{n}}}\ar[r]^{1\otimes f}\ar[d]_{\xi_{i}} & {\cfn X}\ar[d]^{\xi_{i}}\\
{\cfn{\Delta^{n}}\otimes\cfn{\Delta^{n}}}\ar[r]_{f\otimes f} & {\cfn X\otimes\cfn X}
}
\]
commutes for all $i\ge0$.}

If $\iota$ is an inclusion (and $n$ is arbitrary), the conclusion
follows from corollary~4.5 in \cite{smith-steenrod2}. If $n=1$,
and $\iota$ identifies the endpoints of $\Delta^{1}$, there is a
\emph{unique} morphism from $N(\Delta^{1})$ to $\im N(\iota)$ that
sends $N(\Delta^{1})_{1}$ to $\im N(\iota)_{1}$.

If $n=2$, equation~\ref{eq:delta-2-coproduct-1} implies that 
\[
\im(\xi_{0}(\Delta^{2}))=F_{2}\Delta^{2}\otimes F_{0}\Delta^{2}\in\left(N(X)/N(X)_{0}\right)\otimes\left(N(X)/N(X)_{0}\right)
\]
Since corollary~\ref{cor:cf-gives-simplices} implies that $f(\Delta^{2})_{2}=N(\iota)(\Delta^{2})_{2}$,
it follows that the Steenrod-coalgebra morphism, $f$, must send $F_{i}\Delta^{2}$
to $N(\iota)(F_{i}\Delta^{2})$ for $i=0,2$.

Equation~\ref{eq:e1timesdelta2-1} implies that
\[
\im(\xi_{1}(\Delta^{2}))=-F_{1}\Delta^{2}\otimes\Delta^{2}\in N(X)_{1}\otimes\left(N(X)/N(X)_{1}\right)
\]
so that $f(F_{1}\Delta^{2})=N(\iota)(F_{1}\Delta^{2})$ as well.
\end{proof}
We define a complement to the $\cfn *$-functor: 
\begin{defn}
\label{def:fc}Define a functor
\[
\nfc *:\steen\to\dcat
\]
to the category of delta-complexes (see definition~\ref{def:delta-complexes}),
as follows:

If $C\in\steen$, define the $n$-simplices of $\nfc C$ to be the
Steenrod coalgebra morphisms
\[
\ns n\to C
\]
where $\ns n=\cfn{\Delta^{n}}$ is the normalized chain-complex of
the standard $n$-simplex, equipped with the Steenrod coalgebra structure
defined in .

Face-operations are duals of coface-operations
\[
d_{i}:[0,\dots,i-1,i+1,\dots n]\to[0,\dots,n]
\]
with $i=0,\dots,n$ and vertex $i$ in the target is \emph{not} in
the image of $d_{i}$.\end{defn}
\begin{prop}
\label{prop:ux-map}If $X$ is a delta-complex there exists a natural
inclusion
\[
u_{X}:X\to\nfc{\cfn X}
\]
\end{prop}
\begin{rem*}
This is also true if $X$ is an arbitrary simplicial set\foreignlanguage{american}{.}\end{rem*}
\begin{proof}
To prove the first statement, note that any simplex $\Delta^{k}$
in $X$ comes equipped with a map
\[
\iota:\Delta^{k}\to X
\]
The corresponding order-preserving map of vertices induces an Steenrod-coalgebra
morphism 
\[
\cfn{\iota}:\cfn{\Delta^{k}}=\ns k\to\cfn X
\]
so $u_{X}$ is defined by
\[
\Delta^{k}\mapsto\cfn{\iota}
\]
It is not hard to see that this operation respects face-operations.
\end{proof}
So, $\nfc{\cfn X}$ naturally contains a copy of $X$. The interesting
question is whether it contains \emph{more} than $X$:
\begin{thm}
\label{thm:simplicial-complexes-determined}If $X\in\dcat$ is a delta-complex
then the canonical inclusion
\[
u_{X}:X\to\nfc{\cfn X}
\]
defined in proposition~\ref{prop:ux-map} is the identity map on
2-skeleta.\end{thm}
\begin{proof}
This follows immediately from corollary~\ref{cor:n-simplices-map-to-simplices},
which implies that simplices map to simplices and corollary~\ref{cor:cf-gives-simplices},
which implies that these maps are \emph{unique. }\end{proof}
\begin{cor}
\label{cor:cellular-determines-pi1}If $X$ and $Y$ are delta-complexes,
any morphism of their canonical Steenrod coalgebras (see~proposition~\ref{prop:canonical-steenrod-coalgebra})
\[
g:\cfn X\to\cfn Y
\]
induces a map
\[
\hat{g}:X_{2}\to Y_{2}
\]
of 2-skeleta. If $g$ is an isomorphism then $X_{2}$ and $Y_{2}$
are isomorphic as delta-complexes.\end{cor}
\begin{proof}
Any morphism $g:\cfn X\to\cfn Y$ induces a morphism of simplicial
sets
\[
\fc g:\nfc{\cfn X}\to\nfc{\cfn Y}
\]
which is an isomorphism (and homeomorphism) of simplicial complexes
if $g$ is an isomorphism. The conclusion follows from theorem~\ref{thm:simplicial-complexes-determined}
which implies that $X_{2}=\fc{\cfn X}_{2}$ and $Y_{2}=\fc{\cfn Y}_{2}$.
\end{proof}
Propositions~\ref{prop:cx-is-nfx} and \ref{prop:homotopy-equiv-sd-ds}
imply that 
\begin{cor}
\label{cor:pi1-simplicial-set}If $X$ and $Y$ are simplicial sets
and $f:C(X)\to C(Y)$ is a morphism of their canonical Steenrod coalgebras
(see proposition~\ref{prop:canonical-steenrod-coalgebra}) over their
unnormalized chain-complexes, then $f$ induces a map 
\[
\hat{f}:X_{2}\to Y_{2}
\]
of 2-skeleta. If $f$ is an isomorphism, then $\hat{f}$ is a homotopy
equivalence.\end{cor}
\begin{proof}
Simply apply corollary~\ref{cor:cellular-determines-pi1} to $\sd(X)$
and $\sd(Y)$ and then apply $\ds$ and proposition~\ref{prop:homotopy-equiv-sd-ds}
to the map
\[
\hat{f}:\sd(X)_{2}\to\sd(Y)_{2}
\]
that results.
\end{proof}
\bibliographystyle{amsplain}

\providecommand{\bysame}{\leavevmode\hbox to3em{\hrulefill}\thinspace}
\providecommand{\MR}{\relax\ifhmode\unskip\space\fi MR }
\providecommand{\MRhref}[2]{%
  \href{http://www.ams.org/mathscinet-getitem?mr=#1}{#2}
}
\providecommand{\href}[2]{#2}


    \end{document}